\documentclass{article}
\usepackage[utf8]{inputenc}

\title{Finding unavoidable colorful patterns in multicolored graphs}
\author{Matthew Bowen\thanks{McGill University, Montréal, Quebec, Canada. \texttt{matthew.bowen2@mail.mcgill.ca} } \and Ander Lamaison\thanks{Freie Universität and Berlin Mathematical School, Berlin, Germany. \texttt{lamaison@zedat.fu-berlin.de} } \and Alp Müyesser\thanks{Freie Universität and Berlin Mathematical School, Berlin, Germany. \texttt{alp.muyesser@fu-berlin.de} }}
\date{\vspace{-5ex}}
\usepackage{amsmath}
\usepackage{amssymb}
\usepackage{amsthm}
\usepackage{changepage}
\usepackage{enumerate}
\usepackage{ dsfont }
\usepackage{soul}
\usepackage{hyperref}
\usepackage{tikz}
\usepackage[a4paper, total={6in, 8in}]{geometry}

\newcommand{\ep}{\varepsilon}
\newcommand{\expected}{\mathbb{E}}

\theoremstyle{plain}
\newtheorem{theorem}{Theorem}[section]
\newtheorem{lemma}[theorem]{Lemma}
\newtheorem{proposition}[theorem]{Proposition}

\newtheorem{corollary}[theorem]{Corollary}
\newtheorem{definition}[theorem]{Definition}
\newtheorem{question}{Question}

\begin{document}
\maketitle
\begin{abstract}
\par We provide multicolored and infinite generalizations for a Ramsey-type problem raised by Bollob\'as, concerning colorings of $K_n$ where each color is well-represented. Let $\chi$ be a coloring of the edges of a complete graph on $n$ vertices into $r$ colors. We call $\chi$ $\ep$-balanced if all color classes have $\ep$ fraction of the edges. Fix some graph $H$, together with an $r$-coloring of its edges. Consider the smallest natural number $R_\ep^r(H)$ such that for all $n\geq R_\ep^r(H)$, all $\ep$-balanced colorings $\chi$ of $K_n$ contain a subgraph isomorphic to $H$ in its coloring. Bollob\'as conjectured a simple characterization of $H$ for which $R_\ep^2(H)$ is finite, which was later proved by Cutler and Mont\'agh. Here, we obtain a characterization for arbitrary values of $r$, as well as asymptotically tight bounds. We also discuss generalizations to graphs defined on perfect Polish spaces, where the corresponding notion of balancedness is each color class being non-meagre.
\end{abstract}
\section{Introduction}
Graph Ramsey Theory refers to mathematical results that attempt to find large patterns in colored graphs. The pattern in question is most often a monochromatic clique, a complete subgraph whose edges are all the same color. The \textit{Ramsey number} $R(k)$ is the smallest integer $n$ for which every $2$-edge-coloring of the complete graph $K_n$ contains a monochromatic clique on $k$ vertices. It is known that $2^{k/2}\leq R(k)\leq 2^{2k}$, where the constants that appear in the exponents resisted improvements for decades. For an overview, we refer the reader to the 2015 survey by Conlon et al. \cite{Conlon et al.(2015)}.
\par To find patterns that are not monochromatic, we need to assume that all color classes are sufficiently represented in the colorings we consider.
\begin{definition}
    We call $\chi:E(K_n)\rightarrow [r]$ $\ep$-balanced if each color class has at least $\ep\binom{n}{2}$ edges. When it is clear from context, we call $K_n$ $\ep$-balanced if it comes equipped with an $\ep$-balanced coloring.
\end{definition}
\par A result by Erd\H{o}s and Szemerédi shows that graphs which are not $\ep$-balanced contain monochromatic cliques larger than the general bounds would be able to provide \cite{ErdosSzemeredi}. So in some sense, even if we were not interested in finding bi-colored patterns, $\ep$-balanced graphs are the natural graphs to study from a Ramsey theoretic standpoint.
\par The natural function to look at is the \textit{$\ep$-balanced Ramsey number}, which we introduce below. Note that by a color-consistent copy, we mean a subgraph which also preserves the color structure up to relabelling of the colors.
\begin{definition}\label{def:ramsey}
  Let $r\in \mathbb{N}$, fix some $\ep$ with $0<\ep<1/r$, let $H$ be some graph with an associated $r$-edge-coloring. We denote by $R^r_{\ep}(H)$ the smallest $N\in\mathbb{N}$ such that for any $n\geq N$, if $\chi_{\ep}$ is some $\ep$-balanced $r$-edge-coloring of $K_n$, then $K_n$ contains a color-consistent copy of $H$. If no such $N\in\mathbb{N}$ exists, we say $R_{\ep}^r(H)=\infty$.
\end{definition}
When $r$ is not specified, $R_\ep(H)$ denotes $R_\ep^2(H)$. Also, if $\mathcal{H}$ is a family of colored graphs, $R_\ep^r(\mathcal{H})$ denotes the smallest positive integer $N$ for which for all $n\geq N$, all $\ep$-balanced $K_n$ contain a color-consistent copy of some element in $\mathcal{H}$.
\begin{figure}
    \centering
    \includegraphics{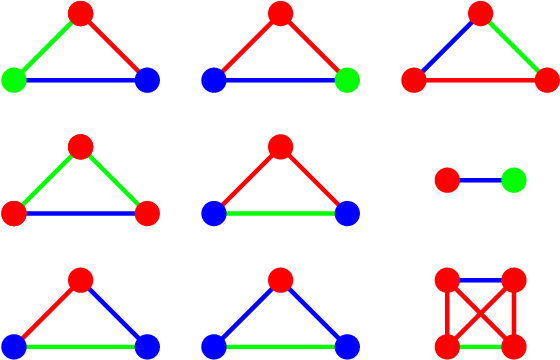}
    \caption{Elements of $\mathcal{F}_k^r$, i.e. the $(r,k)$-unavoidable graphs. Vertices represent complete graphs with $k$ vertices, and edges represent complete bipartite graphs. Theorem \ref{thm:finitetheorem} states that any large enough $\ep$-balanced $3$-colored $K_n$ will necessarily contain one of these graphs, up to permutation of the colors.}
    \label{fig:3colors}
\end{figure}
\par We call a two-coloring of a $K_{2k}$ $k$-unavoidable if one color forms a clique of size $k$, or two disjoint cliques of size $k$. Since $k$-unavoidable graphs are $\ep$-balanced for $\ep$ arbitrarily close to $1/2$, the only $H$ for which $R_\ep^2(H)$ is finite are $H$ which appear in both types of $k$-unavoidable graphs (where $k\geq v(H)$). We will denote the family of $k$-unavoidable graphs as $\mathcal{F}_k$.
\par In the other direction, Bollob\'as conjectured that $R_\ep^2(\mathcal{F}_k)$ is finite, and this was proved by Cutler and Mont\'agh, who showed $R_\ep^2(\mathcal{F}_k)\leq 4^{k/\ep}$ \cite{cutler}. Fox and Sudakov later improved the bound, showing $R_\ep^2(\mathcal{F}_k)\leq (16/\ep)^{2k}$, which is asymptotically tight \cite{stuff}. In particular, this result completely characterizes the two-colored $H$ for which $R_\ep^2(H)$ is finite.
\par A natural question is to similarly classify the $r$-edge-colored $H$ for which $R_\ep^r(H)$ is finite. Our first result is in this direction. To achieve this, we first need to define the right notion of an $(r,k)$-unavoidable graph.
\begin{definition}
  Let $H : = K_{tk}$ be a complete graph on $tk$ vertices whose edges are $r$-colored, for some positive integer $t$. We call $H$ \emph{$(r,k)$-unavoidable} if there is a partition $V(H) = \bigsqcup_{i\in[t]}V_i$ of the vertices of $H$ into $t$ parts, with $|V_i|=k$, such that:
  \begin{enumerate}
     \item For all $i,j\in[t]$, $H[V_i]$ and $H[V_i\times V_j]$ are monochromatic.
      \item All $r$-colors are present in $H$.
      \item Not all $r$ colors are present in $H\setminus V_i$, for any $i\in[t]$.
  \end{enumerate}
 We denote the family of $(r,k)$-unavoidable graphs as $\mathcal{F}_k^r$.
\end{definition}
\par We remark that this definition properly extends the definition of $k$-unavoidable graphs, that is, $\mathcal{F}_k^2=\mathcal{F}_k$. Also, for every $r$, $\mathcal{F}_k^r$ is a finite set. Indeed, any element of $\mathcal{F}_k^r$ can have at most $2r$ vertices, as for every part $V_i$ in the partition, we need to have a color which disappears upon the deletion of $V_i$. And if two parts $V_i$ and $V_j$ share the color that disappears upon their deletion, that color can only appear in the bipartite graph $V_i\times V_j$. Thus, there cannot be three parts which share the color that disappears upon their deletion, which proves the claim. (In fact, any element of $\mathcal{F}_k^r$ can have at most $2r-2$ elements, and this is tight.) See Figure \ref{fig:3colors} for a depiction of all $(3,k)$-unavoidable graphs up to permutation of the colors.
\par We also remark that $\mathcal{F}_k^r$ has size exponentially large in $r$, as one can embed the family of all tournaments on $r$ vertices into $\mathcal{F}_k^r$ by associating to each vertex a different color, and giving the bipartite graphs the color of the vertex they point to in the tournament. We finally remark that there exists an $\ep_r$ for every $r$ such that every element of $\mathcal{F}_k^r$ is at least $\ep_r$-balanced. This follows simply because elements of $\mathcal{F}_k^r$ have their number of parts bounded by a function of $r$.
\par Having established the elementary properties of the family of $(r,k)$-unavoidable graphs, we now state our main result, which is the generalization of Bollob\'as's conjecture \cite{cutler} to arbitrarily many colors.
\begin{theorem}\label{thm:finitetheorem}
For any $r\in\mathbb{N}$ and $\ep$ with $0<\ep<1/r$, there exists a constant $c:=c(r)$ such that for any $k\in\mathbb{N}$, we have that $R_\ep^r(\mathcal{F}^r_k)\leq \ep^{-ck}$.
\end{theorem}
The result is asymptotically tight by a simple probabilistic construction. Further, the result characterizes $r$-edge-colored $H$ with a finite value of $R_\ep^k(H)$ as those which are color-consistent subgraphs of \textit{all} elements of $\mathcal{F}^r_k$ whenever $\ep$ is small enough with respect to $r$ and $k$ is large enough. This is because for every $r$ there exists an $\ep_r$ such that every element of $\mathcal{F}_k^r$ is at least $\ep_r$-balanced.
\par Even though Theorem \ref{thm:finitetheorem} is best possible asymptotically in general, we manage to get better  upper bounds for $R^2_\ep(H)$ for certain ``asymmetric'' $H$. The pattern defined below is a weakening of $\mathcal{F}^2_k$.
\begin{figure}[h]
  \centering
   \hspace*{\fill}
   \begin{tikzpicture}
       \tikzstyle{point}=[circle,thick,draw=black,fill=black,inner sep=0pt,minimum width=4pt,minimum height=4pt]

       \begin{scope}[rotate=17]
       \foreach \number in {1,...,5}{
           \node[point] (N-\number) at ({\number*(360/5)}:1.5cm) {};
       }

       \foreach \number in {1,4,5}{
           \foreach \y in {1,4,5}{
               \draw[blue] (N-\number) -- (N-\y);
           }
       }
       \foreach \number in {2,3}{
           \foreach \y in {1,4,5}{
               \draw[red] (N-\number) -- (N-\y);
           }
       }
       \end{scope}
   \end{tikzpicture}\hfill
    \begin{tikzpicture}
        \tikzstyle{point}=[circle,thick,draw=black,fill=black,inner sep=0pt,minimum width=4pt,minimum height=4pt]

        \begin{scope}[rotate=17]
        \foreach \number in {1,...,6}{
            \node[point] (N-\number) at ({\number*(360/6)}:1.5cm) {};
        }

        \foreach \number in {1,4,5,6}{
            \foreach \y in {1,4,5,6}{
                \draw[red] (N-\number) -- (N-\y);
            }
        }
        \foreach \number in {2,3}{
            \foreach \y in {1,4,5,6}{
                \draw[blue] (N-\number) -- (N-\y);
            }
        }
        \end{scope}
\end{tikzpicture}\hfill
\begin{tikzpicture}
    \tikzstyle{point}=[circle,thick,draw=black,fill=black,inner sep=0pt,minimum width=4pt,minimum height=4pt]

    \begin{scope}[rotate=17]
    \foreach \number in {1,...,7}{
        \node[point] (N-\number) at ({\number*(360/7)}:1.5cm) {};
    }

    \foreach \number in {1,3,4,5,6,7}{
        \foreach \y in {1,3,4,5,6,7}{
            \draw[red] (N-\number) -- (N-\y);
        }
    }
    \foreach \number in {2}{
        \foreach \y in {1,3,4,5,6,7}{
            \draw[blue] (N-\number) -- (N-\y);
        }
    }
    \end{scope}
\end{tikzpicture}
\hspace*{\fill}
    \caption{From left to right, an $M_{2,3}$, an $M_{2,4}$, an $M_{1,6}$}
    \label{fig:M}
\end{figure}
\begin{definition}
  By $M_{l,k}$, we denote a family of $2$-edge-colored graphs with vertex set $V:=L\sqcup R$ where $|L|=l$ and $|R|=k$, and $R$ is a (without loss of generality) red clique, $L\times R$ is a blue complete bipartite graph, and $L$ is an independent set.
\end{definition}
See Figure \ref{fig:M} for clarity. Observe that $M_{0,k}$ is simply a monochromatic set, and $M_{1,k}$ is a particular coloring of a complete graph. The below result shows that when $l$ is treated as a constant, $R_\ep^2(M_{l,k})$ is at most a constant factor away from $R(k)$.
\begin{theorem}\label{thm:asymmetric}
  For any $0<\ep<0.5$, and $l\in\mathbb{N}$, there exists a $C:=C(\ep, l)$ such that $R_{\ep}(M_{l,k})\leq C\cdot R(k)$.
\end{theorem}
This gives a new proof of Bollob\'as's conjecture \cite{cutler} as a $2$-colored $K_n$ containing a $M_{R(k),k}$ necessarily contains an element of $\mathcal{F}_k^2$. The bound we obtain here is worse in general, but has better dependence on $k$. If we use the technique of our Theorem \ref{thm:finitetheorem} or the Fox-Sudakov result, the bound would have to have a multiplicative dependence on $k$ of the form $\ep^{-ck}$ for some positive constant $c$. In contrast, $R(k)\leq 2^{2k}$, with no dependence on $\ep$.
\par We are also interested in explicit values of $\ep$-balanced Ramsey numbers for small patterns. Here, it is natural to fix $\ep$ at $1/2$. The simplest case where the explicit value of $R_{1/2}^2(M_{l,k})$ is non-trivial is when $l=1$ and $k=3$. The Paley graph on $9$ vertices does not contain an $M_{1,3}$, and from the other side one can use our methods to obtain:
\begin{proposition}\label{prop:finiteprop}
    $10\leq R_{0.5}(M_{1,3})\leq 25$
\end{proposition}
We believe it should be possible to close this large gap. Recall that $R(4)=18$, and the unique construction on $17$ vertices without a $4$-clique is a Paley graph.
\bigskip
\par We also consider some infinite analogues of these questions. For graphs defined on the naturals, i.e. $[\mathbb{N}]^2=\{A\subseteq\mathbb{N}: |A|=2\}$, an easy coloring shows that the analogue of Bollob\'as's conjecture does not hold even if we assume every vertex is back degree balanced (i.e. half of the edges from vertex $i$ to vertices in $[i-1]$ are blue and half red) and has infinite degree in both colors. However for graphs defined on $\mathbb{R}$, we see that a simple topological notion of largeness, namely non-meagreness, corresponds well with our notion of $\ep$-balancedness in the finite case.
\begin{theorem}\label{thm:infinitetheorem}
Let $f:[\mathbb{R}]^2\to [r]$ be a coloring such that each color class is non-meagre and has the Baire property (as a subset of $\mathbb{R}^2$). Then $f$ yields some color-consistent copy of $\mathcal{F}^r_\mathfrak{c}$.
\end{theorem}
We give the necessary definitions in Section \ref{sec:infinitestuff}. The proof techniques for Theorems \ref{thm:finitetheorem}
and \ref{thm:infinitetheorem} are very similar, except topological lemmas such as the Localization lemma in Theorem \ref{thm:infinitetheorem} take the place of combinatorial lemmas such at the Dependent Random Choice lemma in Theorem \ref{thm:finitetheorem}.
\vspace{2mm}
\par Throughout the paper, we omit ceiling and floor signs, notably while stating cardinalities of certain sets. As we are concerned only with asymptotics, this does not have an effect on the results.
\section{Proof of Theorem \ref{thm:finitetheorem}}
\par We start with the simple observation that if we can find a complete subgraph $H$ which is the disjoint union of monochromatic cliques of size $k$, the edge set between any two such monochromatic clique is also monochromatic (call $H$ ``blockwise monochromatic'' if it is of this form), and further, all $r$ colors appear somewhere in $H$, then $H$ must contain $(r,k)$-unavoidable graph.
\par We should note that if one is not interested in asymptotically tight bounds, getting a blockwise monochromatic graph which uses all $r$ colors is not difficult. Indeed, a celebrated theorem by K\H{o}vari, S\'os and Tur\'an \cite{KoSoTu} states that if a color class has an $\ep$-fraction of the edge set, then we can find a monochromatic subgraph $K_{s,s}$ where $s=c_\ep \log{n}$ and $c_\ep>0$ depends only on $\ep$. Invoking the K\H{o}vari-S\'os-Tur\'an theorem on each of the $r$ color classes in an $\ep$-balanced $K_n$, we obtain already obtain a subgraph that uses all $r$ colors. Now, we can apply Ramsey's theorem to all $2r$ parts to obtain monochromatic cliques, and afterwards simply use that between any two subset of the vertices, at least one color class has $1/r$-fraction of the edge set, and iteratively invoke the K\H{o}vari-S\'os-Tur\'an theorem to get a smaller subgraph where the edge-set between all $2r$ monochromatic cliques is also monochromatic. This already proves:
\begin{proposition}\label{prop:prop}
  For any $r\in\mathbb{N}$ and $\ep$ with $0<\ep<1/r$, $R_\ep^r(\mathcal{F}^r_k)< \infty$
\end{proposition}
\par In fact, since all elements of $\mathcal{F}^r_k$ are $\ep_r$-balanced for some $\ep_r<1/(4r^2)$ (as any element of $\mathcal{F}^r_k$ has at most $2r-2$ vertices), we can state:
\begin{proposition}
  For any $r\in\mathbb{N}$, there exists some $0<\ep_r<1/(4r^2)$, such that for all $0<\ep<\ep_r$, we have that for $r$-edge-colored graph $H$ on $k$ vertices, $R_\ep^r(H)<\infty$ if any only if for all $F\in\mathcal{F}^r_k$, $F$ contains a color-consistent copy of $H$.
\end{proposition}
When $r=2$, the argument we have sketched above seems to give the simplest proof of Bollobas's conjecture. However, for arbitrary $r$, it gives a tower-type bound where the height of the tower is $\Theta(r)$. On the other hand, our Theorem \ref{thm:finitetheorem}, combined with a simple probabilistic construction, will establish that the correct order of magnitude of $R_\ep^r(\mathcal{F}^r_k)$ at $\ep^{-ck}$, where $c$ is some positive constant that depends only on $r$.
\par However, the layout of the better proof will be very similar in spirit to the simpler argument we provided above. In particular, we will again aim to find a fully-complete multipartite subgraph that uses all $r$ colors. The trick will be to use the dependent random choice technique, which recently proved to be a powerful tool in combinatorics. We refer the reader to the survey by Fox and Sudakov for an overview  \cite{drc}. We should also note that Fox and Sudakov also used the dependent random choice technique to obtain the asymptotically sharp bound on $R_\ep^r(\mathcal{F}^r_k)$ when $r=2$, in \cite{stuff}. However, their method also relied on the fact that in an $\ep$-balanced graph with just two colors, there exists a large subset of vertices which have high degree in red as well as blue. For $\ep$-balanced graphs with more colors (already for $r=3$, see Figure \ref{fig:3colors}), sets with high degree on every color do not necessarily exist.
\par The following lemma is the most basic form of the technique, a proof can be found in \cite{drc}:
\begin{lemma}\label{lem:basicdrc}
  Let $G$ be a graph with average degree $\ep n$. Then, there exists a subset $W$ with $|W|=w$ such that all $k_0$ sized subsets of $W$ have a common neighborhood of size $\beta n$, provided that there exists a positive integer $t$ satisfying:
  $$
    n\ep^t-n^{k_0}\beta^t\geq w
  $$
\end{lemma}
We will need a bipartite version of this lemma, which we state and prove below. We remark that the proof of the below lemma is very similar to the proof of the preceding lemma that can be found in \cite{drc}.
\begin{lemma}\label{lem:bipartitedrc}
  Let $G$ be a graph partitioned into vertex sets $A$ and $B$, with $|A|=m$ and $|B|=m'$, and at least $\ep mm'$ edges between $A$ and $B$. Then, there exists $W\subseteq A$, with $|W|=w$, such that every subset $K\subset W$ with $|K|=k_0$ has $\beta m'$ common neighbors in $B$, provided that there exists a positive integer $t$ satisfying:
  $$
    m\ep^{t}-m^{k_0}\beta^t\geq w
  $$
\end{lemma}
\begin{proof}
  $N_X(v)$ denotes the neighborhood of $v$ in $X$, and $N_X(Y)$ denotes the \textit{common} neighborhood of $Y$ in $X$.
  Sample a subset $T\subseteq B$ of size $t$ uniformly at random with repetition. Let $U= N_A(T)$. Note that a vertex $a\in A$ is in $U$ if and only if $T\subseteq N_B(a)$. Using this fact and linearity of expectation, we calculate:
  $$\expected[\,|U|\,]=\sum_{a\in A}\left(\frac{|N_{B}(a)|}{m'}\right)^{t}\geq \sum_{a\in A}\left(\frac{\ep m'}{m'}\right)^{t} = m\ep^{t},$$
  where the inequality holds by convexity of $x^t$. Now, if we remove from $U$ an element from every $k_0$-element subset $R$ which fails to have $|N_{B}(R)|\geq \beta m'$, the resulting set will have the desired properties. Let $Z$ be the random variable denoting how many such $k_0$-element subsets are there in $U$.
  $$\expected\left[Z\right]=\sum_{\substack{R\in\binom{A}{k_0}\\N_{B}(R)<\beta m'}}\left(\frac{N_{B}(R)}{m'}\right)^t\leq\sum_{\substack{R\in\binom{A}{k_0}\\N_{B}(R)<\beta m'}}\beta^t\leq m^{k_0}\beta^t$$
Thus, on average, we would have to remove at most  $m^{k_0}\beta^t$ vertices to modify $U$ to create a set $W$ with the desired properties (removing one vertex from each ``bad'' set $R$). Using linearity of expectation one last time, we derive $\expected[W]\geq\expected[\,|U|-Z]\geq m\ep^{t}-m^{k_0}\beta^t$, so in particular, a $W$ with at least this size must exist, as claimed. \end{proof}
The below corollary simply follows by iterating the above argument, and will be convenient when we establish connections between multipartite graphs. If the edges of a graph are colored, define $N_X^c(Y)$ to be the set of common neighbors of $Y$ in $X$ through edges of color $c$.
\begin{corollary}\label{cor:multipartitedrc}

   Let $r$ and $y$ be positive integers, $r\geq 2$. There exists $N=N(r,y)$ with the following: let $A, B_1, \dots, B_y$ be disjoint vertex sets, all of size $n>N$. Consider an $r$-coloring of the edges of the complete bipartite graph between $A$ and each $B_i$. Then there exists a set $W\subseteq A$, $|W|=n^{2^{-y}}$, and colors $c_1, \dots, c_y$ such that for every $1\leq i\leq y$ and every subset $X\subseteq W$ of size $\frac18\log_rn$ we have $|N_{B_i}^{c_i}(X)|\geq \sqrt{n}$.
\end{corollary}

\begin{proof} We will construct a nested sequence of sets $A=A_0\supseteq A_1\supseteq\dots\supseteq A_y=W$, with $|A_i|=n^{2^{-i}}$ such that every $X\subseteq A_i$ of size at least $\frac18\log_rn$ has $N_{B_i}^{c_i}(X)\geq \sqrt{n}$. The definition of $A_0$ is clear. As an induction hypothesis, suppose that $A_{i-1}$ has been defined. Let $c_i$ be the most common color in the bipartite graph between $A_{i-1}$ and $B_i$. Apply Lemma \ref{lem:bipartitedrc} to the graph formed by the edges of color $c_i$, with the values $m=n^{2^{-(i-1)}}$, $m'=n$, $\epsilon=r^{-1}$, $w=n^{2^{-i}}$,  $k_0=\frac18\log_rn$, $\beta=\frac1{\sqrt{n}}$, $t=\frac{2^{-(i-1)}}{3}\log_rn$, and denote by $A_i$ the resulting set $W$. This is possible for $n$ large enough (say $n>N_i$): \[m\epsilon^t-m^{k_0}\beta^t=n^{\frac{2\cdot 2^{-(i-1)}}{3}}-n^{-\frac{2^{-(i-1)}}{24}\log_rn}>n^{2^{-i}}\] This proves the statement, with $N=\max N_i$.\end{proof}

\begin{corollary}\label{cor:finaldrc} Let $r$ and $t$ be positive integers, $r\geq 2$. There exists $N=N(r,t)$ with the following: let $n\geq N$ and $A_1, A_2, \dots, A_t$ be disjoint subsets of the vertex set of an $r$-colored complete graph of size $|A_i|=n$. Then there exist subsets $X_i\subset A_i$, of size $|X_i|=\frac{1}{2^{t+1}r}\log_rn$, such that every set $X_i$ is monochromatic and every complete bipartite graph between $X_i$ and $X_j$ is monochromatic.\end{corollary}

\begin{proof} Induction on $t$. For $t=1$, this follows from the multicolor version of Ramsey's theorem. Suppose that the result is true for $t=i-1$, and we will prove it for $t=i$. Apply Corollary \ref{cor:multipartitedrc} with $y=i-1$, $A=A_i$ and $B_j=A_j$ to obtain a set $W\subset A_i$ and colors $c_1,\dots, c_{i-1}$ with the properties of Corollary \ref{cor:multipartitedrc}, as long as $n>N_1$, where the value of $N_1$ is given by Corollary \ref{cor:multipartitedrc}. By Ramsey's theorem, $W$ contains a monochromatic set $X_i\subset W$ of size $\frac{1}{4r}\log_r|W|=\frac{1}{2^{i+1}}\log_rn$.
  \par By our hypothesis, for every $1\leq j\leq i-1$ we have $\left|N_{A_j}^{c_j}(X_i)\right|\geq \sqrt{n}$. Let $A'_j\subset N_{A_j}^{c_j}(X_i)$ be subsets of size $\sqrt{n}$. By the induction hypothesis, there exists $N_2$ such that, if $\sqrt{n}>N_2$, there exist subsets $X_j\subset A'_j$ such that each of them is monochromatic and the bipartite graph between every pair of them is monochromatic. Their size is $|X_j|=\frac{1}{2^ir}\log_r\sqrt{n}=\frac{1}{2^{i+1}r}\log_rn$. This proves the statement with $N=\max\{N_1, N_2^2\}$.\end{proof}

\par Note that for $t=r$ we have $|X_i|=C_r\log n$. We are now ready to start the proof of the main theorem.

\begin{proof}[Proof of Theorem \ref{thm:finitetheorem}]  As mentioned to previously, our goal is to find a fully-complete multipartite graph, which uses all $r$ colors, in a large enough $\ep$-balanced $K_n$. In the first step of the proof, we will apply Lemma \ref{lem:basicdrc} to the $r$ graphs induced by the edges colored in the $r$ different color classes. The subsets we collect via the Lemma here will thus necessarily utilize all $r$ colors. Afterwards, we will apply Corollary \ref{cor:finaldrc} to fill in for the connections between the various subsets we collected in the previous step. We now give the details. We assume $n>\ep^{-c_rk}$ where $c_r$ is a sufficiently large constant that only depends on $r$ which will be specified later.
\par Given a color $i\in[r]$, we apply Lemma \ref{lem:basicdrc} with parameters $w=\sqrt{n}$, $k_0=\frac{-\log_\ep n}{8}$, $\beta=\frac{1}{\sqrt{n}}$ and $t=\frac{-\log_\ep n}3$. Thus for each color we obtain a set $W_i$ of size $|W_i|=\sqrt{n}$. These sets are not necessarily disjoint, so we take disjoint subsets $W'_i\subset W_i$ of size $|W'_i|=\frac{\sqrt{n}}{r}$. We can now apply Corollary \ref{cor:finaldrc} to find monochromatic subsets $X_i\subset W'_i$ of size $|X_i|=\frac{1}{2^{r+1}r}\log_r|W'_i|=\frac{1}{2^{r+1}r}\left(\frac12\log_rn-1\right)\geq -\frac{1}{2^{r+3}r}\log_{\ep}n$, pairwise joined by monochromatic graphs. Take a subset $X'_i\subset X_i$ of size $|X'_i|= -\frac{1}{2^{r+3}r}\log_{\ep}n$.

\par Since $|X'_i|\leq \frac{-\log_\ep n}{8}$, we have, by Lemma \ref{lem:basicdrc}, that $|N^i(X'_i)|\geq\sqrt{n}$, where $N^i(\cdot)$ denotes the common neighborhood of $\cdot$ in color $i$. Let $U_i$ be a subset of this common neighborhood of size $\sqrt{n}$. Take subsets $U'_i\subseteq U_i$ of size \begin{equation}\label{eq:ugly}\frac{\sqrt{n}-r\frac{1}{2^{r+3}r}\log_\ep n}{r}\geq \sqrt[3]n \end{equation} which are pairwise disjoint and disjoint from all $X'_i$. Note that the inequality follows because for a large enough choice of $c_r$, $n$ will be sufficiently larger than $r$.
\par We claim that we can take subsets $U''_i\subset U_i'$ of size $\sqrt[4]n$ with the following property: for every $v, v'\in U''_i$ and every $w\in\cup_{j=1}^rX'_j$, the edges $vw$ and $v'w$ have the same color. This holds by the pigeonhole principle; one could associate a base-$r$ vector $\vec{v}_i$ of length $\left|\bigcup\limits_{j=1}^rX'_j\right|$ to each $u_i\in U_i$ where $\vec{v}_i(j)$ denotes the color of the edge $u_i$ sends to the $j^{th}$ vertex in $\left|\bigcup\limits_{j=1}^rX'_j\right|$, and therefore there exists a subset $|U_i''|$ with: \begin{equation*}\frac{|U'_i|}{r^{\left|\bigcup\limits_{j=1}^rX'_j\right|}}\geq \frac{\sqrt[3]n}{r^{\frac{1}{2^{r+3}}\log_rn}}=n^{\frac13-\frac{1}{2^{r+3}}}\geq\sqrt[4]n\end{equation*} elements that all have the same associated vector. Notice that at this point we know that the edges between each $w\in X_j$ and all of $U_k''$ are monochromatic. We wish that the color of these edges do not depend on the choice of $w$.
\par To achieve this, we associate to each vertex $w_i\in X'_j$ a base-$r$ vector $\vec{v}'_i$ of length $r$, where $\vec{v}'_i(h)$ is defined to be the unique color of the edges between $w_i$ and $U_h''$. By the pigeonhole principle, we can find a subset $X''_j\subset X'_j$ of size $\frac{|X'_j|}{r^r}$ where all vertices have the same associated vector. The bipartite graph between $X_j''$ and $U_h''$ is therefore monochromatic, for all $j,h\in [r]$.
\par Finally, apply Corollary \ref{cor:finaldrc} to the sets $U''_i$ to produce monochromatic sets $Y_i\subseteq U_i''$ of size $|Y_i|\geq\frac{1}{2^{r+1}r}\log_r|U''_i|\geq-\frac{1}{2^{r+3}r}\log_\ep n$ pairwise joined by monochromatic graphs. The graph induces on the vertex sets $Y_i$ and $X''_i$ satisfy the properties that we want: \begin{itemize}
\item The sets $Y_i$ and $X''_i$ are monochromatic.
\item The bipartite graphs $Y_iY_j$ and $X''_iX''_j$ are monochromatic (by Corollary \ref{cor:finaldrc}), as well as the bipartite graphs $Y_iX''_j$ (by the pigeonhole principle).
\item Every color appears in the graph, in particular the bipartite graph $Y_iX''_i$ has color $i$.
\end{itemize}
\par Since $|Y_i|\geq -\frac{1}{2^{r+3}r}\log_\ep n$ for every $i$ and $|X_j''| \geq -\frac{1}{2^{r+3}r^{r+1}}\log_\ep n$ for every $j$, we can choose any $n$ such that $n>\ep^{2^{r+3}r^{r+1}k}$ and $n$ is large enough to make inequality \eqref{eq:ugly} true.
\end{proof}
\section{Asymmetric Patterns}
\par In this section, we specialize on getting upper bounds on the function $R_{\ep}(M_{l,k})$.
\vspace{2mm}
\par Given subsets $A$ and $B$ of a $2$-colored graph $G$, we say $A$ is \textit{complete with} $B$ in red (blue) if all the edges in between $A$ and $B$ are red (blue). If $A=\{a\}$, we say $a$ is complete with $B$ in red (blue). We say \textit{red (blue) neighborhood} of a vertex to mean the subset of vertices which to vertex is adjacent to via red (blue) edges. We say the \textit{red (blue) degree of a vertex} to mean the size of its red (blue) neighborhood.
\par The upper bound in Proposition \ref{prop:finiteprop} follows from optimizing the below lemma for small integers.
\begin{lemma}\label{lem:finitelem}
  Let $K_n$ be an $\ep$-balanced graph, where $\ep$ is as large as possible. Then, there exists a subset of vertices $S$ with $|S|\geq cn$ such that for some $x,y\in K_n$, $S$ is complete with $x$ in red and $S$ is complete to $y$ in blue. Further, $c\geq \sqrt{1-\ep}-(1-\ep) + o(1)$
\end{lemma}
\begin{proof}
Let $K_n$ be a $\ep$-balanced graph, where $\ep$ is as large as possible. Without loss of generality, $K_n$ contains at most as many red edges as blue edges. Let $\Delta_R$ and $\delta_R$ denote the maximum and minimum red degrees of $K_n$. Also, let $c:=\sqrt{1-\ep}-(1-\ep)$.
\par \textbf{Case 1: } $\Delta_R-\delta_R> cn$. In this case, we simply consider the vertices $x$ and $y$ of maximum and minimum red degree, and observe that $y$ must send at least $cn$ blue edges into the red neighborhood of $x$.
\par \textbf{Case 2: $\Delta_R-\delta_R\leq cn$}. Since the average degree is $\ep n$, we must have $\Delta_R \leq (\ep + c)n$ and $\delta_R \geq (\ep - c)n$. We count the number of paths of length $2$ where one edge is colored red, and the other blue. More precisely, we count the sets of the form $\{x,y,z\}\subseteq V(K_n)$, where $\{x,y\}$ is red and $\{y,z\}$ is blue. Let $d_R(v)$ denote the red degree of $v$. We call such sets rainbow cherries. Then, the number of rainbow cherries is:
\begin{align*}
  \sum_{v\in K_n}d_R(v)(n-1-d_R(v)) &= \sum_{v\in K_n} \left(\frac{n-1}{2}\right)^2-\left(\frac{n-1}{2}-d_R(v)\right)^2\\
  &\geq \sum_{v\in K_n} \left(\frac{n-1}{2}\right)^2-\left(\frac{n-1}{2}-(\ep-c)n\right)^2\\
  &=\sum_{v\in K_n} (\ep-c)(1-\ep+c)n^2 + o(n^2)\\
  &= (\ep-c)(1-\ep+c)n^3 + o(n^3)
\end{align*}
\par As there are at most $n^2$ pairs of vertices, for a particular pair $\{x,z\}$, it must be that there exists at least $(\ep-c)(1-\ep+c)n+o(n)$ other $y$ such that $\{x,y,z\}$ is a rainbow cherry.
Then, for this pair $\{x,z\}$, the intersection of their red and blue neighborhood at least $(\ep-c)(1-\ep+c)n+o(n)$. Since $c= (\ep-c)(1-\ep+c)$ by choice of $c$, the statement of the lemma follows.
\end{proof}
\par We remark that the dependence of $c$ on $\ep$ given in the previous lemma is not optimal. The right dependence is $c=\ep(1-\ep) + o(1)$, which can be proven with a more elaborate counting argument, but we don't include this in the present paper. The sharpness of this constant can be seen by considering a random graph where an edge is colored red with probability $\ep$ and blue otherwise.
\par We now prove the following lemma, which is a generalization of the previous one that will allow us to prove Theorem \ref{thm:asymmetric}.
\begin{lemma}\label{lem:coloronside} For every $\epsilon>0$ and $k$ there exists $c(\epsilon,k)>0$ with the following property: every $\epsilon$-balanced $K_n$ has a subset $S\subseteq V(G)$, $|S|\geq cn-O(1)$ and $A,B\subseteq V(K_n)$ with $|A|,|B|\geq k$ and $A$ is complete with $S$ in red and $B$ is complete with $S$ in blue.
\end{lemma}
\begin{proof}[Proof of Lemma]
Let $t>0$ be a constant such that $0<\frac{1-t}{2k}<\epsilon$. Let $\alpha:=\frac{1-t}{2k}$. Let $r'_1, \dots, r'_k$ and $b'_1, \dots, b'_k$ be the $k$ vertices with the largest red degree and blue degree, respectively. We consider two cases:
\\

\textbf{Case 1:} $\sum_{i=1}^k\left(d_B(r'_i)+d_R(b'_i)\right)\leq (1-t)n$. Then the set $S$ of vertices joined by a red edge to each $r'_i$ and to a blue edge to each $b'_i$ has size at least $tn-2k$. Indeed, the sum bounds from above the number of vertices in $V(G)\setminus (\{r_1',\cdots, r_k'\} \cup \{b_1',\cdots, b_k'\})$ which are adjacent to an $r_i'$ in blue or to a $b_i'$ in red.
\\

\textbf{Case 2:} $\sum_{i=1}^k\left(d_B(r'_i)+d_R(b'_i)\right)\geq (1-t)n$. There is some $i$ such that $d_B(r'_i)\geq \alpha n$ or $d_R(b'_i)\geq\alpha n$. Without loss of generality, assume the former. Then, we have that $d_B(v)\geq \alpha n$ for all but at most $k-1$ vertices $v\in V(G)$.
\\

Let $V_R$ be the set of vertices with $d_R(v)\geq \alpha n$. We can use that $G$ is $\ep$-balanced to find a bound on $|V_R|$:\[2\epsilon{n \choose 2}\leq 2|E_R(G)|=\sum\limits_{v\in V(G)}d_R(v)=\sum\limits_{v\in V_R}d_R(v)+\sum\limits_{v\notin V_R}d_R(v)\leq n|V_R|+\alpha n\left(n-|V_R|\right)\] which rearranges to $|V_R|\geq \frac{\epsilon-\alpha}{1-\alpha}n-\frac{\epsilon}{1-\alpha}$.
\\

Now consider the set $S'$ of $(2k+1)$-tuples of distinct vertices $(v,r_1, \dots, r_k, b_1, \dots, b_k)$ such that, for every $i$, the edge $vr_i$ is red and the edge $vb_i$ is blue. If we fix $v\in V_R\setminus\{r'_1,\dots,r'_k\}$, then the number of choices of $r_1, \dots, r_k, b_1, \dots, b_k$ is $(d_R(v))_k(d_B(v))_k\geq ((\lfloor\alpha n\rfloor)_k)^2$, where $(x)_k$ denotes the falling factorial $x(x-1)\cdots (x-k+1)$. We deduce that $|S'|\geq \left(\frac{\epsilon-\alpha}{1-\alpha}n-\frac{\epsilon}{1-\alpha}-k\right)((\lfloor\alpha n\rfloor)_k)^2$. By the pigeonhole principle, there is a choice of $r_1, \dots, r_k, b_1, \dots, b_k$ for which there are at least $\left(\frac{\epsilon-\alpha}{1-\alpha}n-\frac{\epsilon}{1-\alpha}-k\right)((\lfloor\alpha n\rfloor)_k)^2n^{-2k}$ choices of $v$, which we can put in a set $S$.
\\

Since $\epsilon$ and $k$ do not depend on $n$, this proves the statement for $c=\min\{t,\frac{\epsilon-\alpha}{1-\alpha}\alpha^{2k}\}$.
\end{proof}
\begin{corollary}For every $\ell,r$ and $\epsilon$ there exists $N$ with the following property: every $\epsilon$-balanced $K_n$ on $n>N$ vertices contains disjoint sets $A,B,S$ with $|A|=|B|=\ell$, $|S|=k$, $A$ is complete with $S$ in red and $B$ is complete with $S$ in blue. Moreover, if $\epsilon$ is fixed, then $N\leq C_\epsilon k(4\ell)^{2\ell}$. \end{corollary}
\begin{proof} If $\epsilon\ell>1$, follow the proof above by choosing $t=1/2$ and $\alpha=\frac1{4\ell}$. Then $0<\alpha<\epsilon$. For a value of $n>C_1k(4\ell)^{2\ell}$, where $C_1$ is large enough, we have \[k\leq\min\left\{tn-2\ell, \left(\frac{\epsilon-\alpha}{1-\alpha}n-\frac{\epsilon}{1-\alpha}-\ell\right)((\lfloor\alpha n\rfloor)_\ell)^2n^{-2\ell}\right\}.\]

If $\epsilon\ell\leq1$, then applying Lemma \ref{lem:coloronside} we obtain new constant $C_2:=c(\varepsilon, \lfloor\varepsilon^{-1}\rfloor)$ such that for $n> C_2r$ there exist $A'$, $B'$ and $S$ with $|A'|=|B'|=\lfloor\epsilon^{-1}\rfloor$, $|S|=r$ and the desired properties. Simply take any $A\subseteq A'$, $B\subseteq B'$ of size $\ell$. This proves the statement for $C_\epsilon=\max\{C_1,C_2\}$. \end{proof}

Observe that Theorem \ref{thm:asymmetric} follows immediately from the above Corollary replacing $r$ with $R(r)$.

\section{Infinite Balanced Graphs}\label{sec:infinitestuff}

We now begin considering what natural restrictions we can put on our colorings to generalize our previous results to the infinite case.  The main difficulty here is that it is no longer clear how the notion of being $\ep$-balanced should generalize. For example, we cannot obtain an $M_{\omega, \omega}$ (a countably infinite analogue of an $M_{l,r}$) by merely assuming every vertex has infinite degree in both colors.  Even under the stronger assumption that the back-degrees are balanced (i.e. half of the edges from vertex $i$ to vertices in $[i-1]$ are blue and half red) this is still not enough. The following construction applies to both restrictions.

\begin{proposition}
There exist back-degree balanced $2$-colorings of the complete graph $K$ on $\omega$ vertices such that every vertex has infinite red and blue degree yet $K$ contains no $M_{1,\omega}$.
\end{proposition}

\begin{proof}
Consider the graph with vertex set consisting of two disjoint copies of $\mathbb{N}$.  Color all edges between vertices in the right copy blue and all edges between vertices in the left copy red.  For each vertex $i$ in the left copy and $j$ in the right copy, color edge $ij$ blue if $i<j$ and red otherwise.  Notice that this coloring has an $M_{l,r}$ for any $l,r\in\mathbb{N}$ yet no $M_{1,\omega}$. To see that we may choose this coloring to be back-degree balanced, one may start with evens and odds as the partitions.
\end{proof}

On the other hand, if we deal with the graphs defined on Polish spaces the notion of being non-meagre seems to correspond rather well with the density results from the finite version.  For convenience, we recall the necessary definitions and results. The reader may see \cite{kechris} for a more extensive overview. $X$ and $Y$ will always denote topological spaces in the below.

\begin{definition}
We say $X$ is a perfect Polish space if it is separable, completely metrizable, and has no isolated points.
\end{definition}
\begin{definition}
A set $E\subseteq X$ is nowhere dense if its closure has empty interior. We say $E$ is meagre if it is the union of countably many nowhere dense sets, and $E$ is comeagre if $X\setminus E$ is meagre.
\end{definition}
\begin{definition}
  $E\subseteq X$ has the Baire property (BP) if there exists an open set $O$ such that $E\triangle O$ is meagre, and we call $f:X\to Y$ Baire measurable if the preimages of open sets have the BP.
\end{definition}

We now restate Theorem \ref{thm:infinitetheorem} in more generality:
\begin{theorem}\label{thm:inf}
  Suppose $X$ is a perfect Polish Space and  $f:[X]^2\rightarrow [r]$ is Baire measurable and each color class is non-meagre. Then we find some color consistent member of  $\mathcal{F}_{\mathfrak{c}}^r$ where each blown-up clique is a Cantor set.
\end{theorem}

The following easy consequences of the above definitions will be useful in the proof. Proofs of the following two lemmas can be found in Chapter 8 of \cite{kechris}:

\begin{lemma}[Localization]
Suppose $A\subseteq X$ has the $BP$.  Then either $A$ is meagre or there exists a non-empty open subset $U$ on which $A$ is comeagre.
\end{lemma}

\begin{lemma}[Continuous Restriction]
Let $X,Y$ be Polish spaces and $f: X\rightarrow Y$ Baire measurable.  Then $f$ is continuous on some countable intersection of dense open sets.
\end{lemma}
\par Recall that using the axiom of choice \cite{sier}, one can define graphs on real numbers that fail to have a continuum sized monochromatic clique. Galvin's theorem \cite{Galvin}, stated below, allows us to find a monochromatic clique which is a Cantor set (which is in particular continuum sized), if all the color classes have the BP.

\begin{theorem}[Galvin]\label{lem:galvin}
Suppose $X$ is a perfect Polish space, and $[X]^2=P_0\cup...\cup P_i$ is a partition where each $P_i$ has the BP (as subsets of $X^2$).  Then there is a Cantor set $C\subseteq X$ such that $[C]^2\subseteq P_i$ for some $i$.
\end{theorem}

Now, before beginning the proof of Theorem \ref{thm:inf} we outline the main idea, emphasizing the similarities to the proof of Theorem \ref{thm:finitetheorem}.  Just as Theorem \ref{lem:galvin} generalizes Ramsey's theorem, the KST theorem admits a similar generalization: if $A\subseteq [X]^{2}$ is non-meagre, then we can find Cantor sets $C_1,C_2$ such that $C_1\times C_2\subseteq A$.  Using this, one could hope that we could replicate the proof of Proposition \ref{prop:prop} given at the start of section 2.  However, recall that the bounds given in that proof were rather poor.  This issue is more substantial in the BP setting: after running the analogue of KST even a single time we are left with meagre sets, where the coloring function can be too poorly behaved to continue.  In the finite setting we were able to get around this issue by utilizing Corollary \ref{cor:multipartitedrc}, which allowed us to do many restrictions simultaneously.  Here we will employ a similar idea.  In particular, rather than using a variation of KST directly we will instead prove a generalisation of it along the lines of Corollary \ref{cor:multipartitedrc}.  Moreover, we will ensure that the coloring function is continuous on the Cantor sets we obtain, which will allow us to apply Theorem \ref{lem:galvin} despite these sets being meagre.

\begin{proof}[Proof of Theorem \ref{thm:inf}]

Analogously to beginning with monochromatic bipartite graphs in the finite case, we start by localizing to disjoint open sets $U_1,...,U_{2r}$ such that $U_{2k-1}\times U_{2k}$ is comeagre in color $k$ for $k\in [r]$.   We note that this step preserves structure: for each $i,j\in [2r]$ and $k\in[r]$, $U_i\times U_j$ is open in $X^2$ and $(U_i\times U_j)\cap f^{-1}(k)$ has the BP (in $U_i\times U_j$).  Therefore, if necessary, we can localize again and rename the sets to also ensure each $U_i\times U_j$ is comeagre in some color.

\vspace{3mm}

Now, we begin collecting an assortment of well-behaved subsets within and between each $U_i$. As for each $i\in[2r]$, $U_i$ is open, $f|[U_i]^2$ is still Baire measurable. So, we can find countable intersections of dense open subsets $G_1:=\bigcap_{m\in\mathbb{N}} G_1^m$,$\,\cdots\,$,$\,\,G_{2r}=\bigcap_m G^m_{2r}$ where $f$ is continuous.  Further, for every $i,j\in[2r]$ with $i\neq j$, $U_i\times U_j$ is comeagre in some color, so there is a sequence of dense open sets $S_{i,j}^m$ for every $m\in\mathbb{N}$ such that $\bigcap_{m\in \mathbb{N}}S_{i,j}^m\subseteq U_i\times U_j$ is monochromatic.  We note that as each $G_i^m$ and $S_{i,j}^m$ are dense and open, for any $m\in \mathbb{N}$ and open sets $V_1,\cdots,V_{2r}$ there are restricted open sets $V_1',\cdots,V_{2r}'$ such that each $V_i'\subseteq G^m_i$, and $V_i'\times V_j'\subseteq S_{i,j}^m$.

\vspace{3mm}

Using these sequences of open sets, we will define Cantor sets $C_1,\cdots,C_{2r}$ with $C_i\subseteq U_i$, being careful to ensure the edges in $C_i\times C_j$ are monochromatic and that $f|[C_i]^2$ is continuous for any $i,j\in [2r]$. The construction method we will use is commonly referred to as a Cantor scheme.
\vspace{3mm}
\par We first fix some compatible metric $d$ on $X$ and set $R_i^\emptyset=U_i$.  Once we have already defined $R_i^s$ for $s\in \{0,1\}^n$, we will define $R_i^{s^\frown l}$ for $l\in \{0,1\}$ ensuring that the following hold:
\begin{enumerate}
\item $R_i^{s^\frown 0}$ and $R_i^{s^\frown 1}$ are non empty disjoint open sets.
\item $cl(R_i^{s^\frown l})\subseteq R_i^s$ and $diam(R_i^{s^\frown l})\leq 2^{-n-1}$.
\item $R_i^{s^\frown l}\subseteq G^{n+1}_i$.
\item For any $i,j\in[2r]$ and $u,v\in \{0,1\}^{n+1}$, $R_i^u\times R_j^v\subseteq S_{i,j}^{n+1}$.
\end{enumerate}
The discussion in the second paragraph guarantees that we can satisfy these requirements at every step.

We let $C_i=\bigcap_{n\in\mathbb{N}}\bigcup_{s\in 2^n} R_i^{s}$ (which is well defined by the completeness of $X$ and $(2)$) denote the Cantor set associated with the sequence $R_i$.  As $f|[C_i]^2$ must still be continuous since $C_i\subseteq G_i$ by $(2,3)$, we may apply Galvin's theorem within each $C_i$ to obtain the desired configuration.
\end{proof}

\section{Discussion}
\par Here, we we collect some open problems and future directions of research. Firstly, even though the upper bound provided by Theorem \ref{thm:finitetheorem} is asymptotically tight with respect to $\ep$ and $k$, the constant that arises that depends only on the number of colors $c(r)$, was rather large. It would be interesting to obtain an upper bound of the form $\ep^{-ck}$, where $c$ is sub-exponential in $r$ (the number of colors).
\par An extension of Bollob\'as's conjecture for hypergraphs was given in \cite{stuff} (Theorem 4.2). It would not be difficult to extend this result to arbitrary many colors, using an appropriate family of multicolored unavoidable hypergraphs as we did here for multicolored graphs. However, as far as we are aware there are no practical bounds, already for two colors and uniformity three. We state the result for uniformity three below informally. A more precise statement for arbitrary uniformities can be found in \cite{stuff}.
\begin{theorem}
For any $\ep>0$ and positive integer $k$, for sufficiently large $n$, any two-coloring of $K_n^{(3)}$ (the $3$ uniform complete graph on $n$ vertices) with $\ep n^3$ edges in both colors contains a subgraph on $3k$ vertices with three disjoint sets of vertices of size $k$ such that the color of any $3$-edge depends only on the sizes of the intersections between the edge and each of the three parts and red and blue both appear somewhere in the graph.
\end{theorem}
\par A proof for the above Theorem in full generality was given in the appendix of \cite{kwan}, as it was used to prove another result in that paper. However, no explicit bound on $n$ was cited, and from the proof it is clear that the dependence of $n$ on $r$ (uniformity) is a tower of length much larger than $r$. For $r=3$, as stated in the above Theorem, we pose the following problem that was essentially also raised in \cite{stuff}.
\begin{question}
Can we take $n\geq (1/\ep)^{2^{ck}}$ in Theorem 5.1, for $c$ some absolute constant?
\end{question}
\par For the convenience of the reader, we sketch a short proof of Theorem 5.1 that is somewhat different than the proof that was given in \cite{kwan}, which we believe is more direct. Like in \cite{kwan}, we will not get any sensible bounds because of a reference to the Product Ramsey Theorem (\cite{promel}, Theorem 9.2) that is a greedy iterated application of  hypergraph Ramsey theorem.
\begin{proof}[Sketch of Theorem 5.1]
Consider an $\ep$-balanced $3$-uniform complete hypergraph on $N$ vertices where $N$ is sufficiently large. As each color has $\Omega(n^3)$ edges, by a result of Erd\H{o}s (\cite{erdos}, Theorem 1) (that is a generalization of the K\H{o}vari-S\'os-Tur\'an Theorem \cite{KoSoTu} for hypergraphs) we can find a red monochromatic tripartite graph (vertex set partitioned into three parts such that the edges intersect each part in exactly one vertex), such that the sizes of each part is $N_1$, where $N_1$ depends only on $N$ and is sufficiently large. We can similarly find a large blue tripartite graph that is disjoint from the red tripartite graph.
\par We can now apply the Product Ramsey Theorem (\cite{promel}, Theorem 9.2) to the red tripartite graph until edges that intersect the same parts in the same amount are monochromatic. If blue appears somewhere in the resulting substructure, we are done, assuming $N_1$ was sufficiently large. Otherwise, we have found a large red clique. Applying the same argument to the blue tripartite graph, we may assume we found a large blue clique.
\par Now we may apply the Product Ramsey Theorem between the large red clique and the large blue clique, which will yield a structure of the desired type.
\end{proof}
\par We also note that in the infinite case we considered only one of many potential ways of generalizing the notion of balanced colorings. Another natural question is as follows:
\begin{question}
If $c:[\mathbb{R}]^2\rightarrow[r]$ is a Lebesgue measurable coloring such that each color class is non-null, can we find some color consistent member of  $\mathcal{F}_{\mathfrak{c}}^r$ where each blown-up clique is a Cantor set?

\end{question}

As in the BP setting there are analogues of the Ramsey and KST theorems that suggest such a result should hold.  We can even prove an analogue of Corollary \ref{cor:multipartitedrc} in this setting, which was the key step in the proof of Theorem \ref{thm:inf}.  In particular, we can adapt the proof of the Brodski-Eggleston Theorem \cite{brod, egg} to find a graph with the same multipartite structure as a member of an $\mathcal{F}_{\mathfrak{c}}^r$. However, unlike in the BP setting, we are not currently able to ensure that the coloring function is continuous on the Cantor sets we end up with after this step, so we are unable to apply a variant of Ramsey's theorem and complete the proof.
\section{Acknowledgements}
We would like to thank Benny Sudakov for bringing to our attention the results in \cite{cutler} and \cite{stuff}. We also thank Clinton Conley, Shagnik Das, Wesley Pegden, and Tibor Szabó for helpful discussions.


\begin{thebibliography}{9}
  \bibitem{brod}
  M. L. Brodskii, \textit{On some properties of sets of positive measure} (Russian),Uspekhi Mat. Nauk.4, No. 3,31(1949), 136-139.

   \bibitem{Conlon et al.(2015)}D. Conlon, J. Fox, \& B. Sudakov, \textit{Recent developments in graph Ramsey theory}, 2015, arXiv:1501.02474
  \bibitem{cutler}
  J. Cutler and B. Montagh, \textit{Unavoidable subgraphs of colored graphs}, Discrete Math. 308 (2008),
  4396–4413.
  \bibitem{diestel}
   R. Diestel, Graph Theory, second ed., Springer, 1997.
   \bibitem{egg}
   H. G. Eggleston, \textit{Two measure properties of Cartesian product sets}, Quart. J. Math. Oxford(2)5(1954), 108-115.
  \bibitem{erdos}
  P. Erdős, On extremal problems of graphs and generalized graphs, Israel J. Math. 2 (1964), no. 3, 183–190.
  \bibitem{ErdosSzemeredi}
  P. Erd\H{o}s and E. Szemerédi, \textit{On a Ramsey type theorem}, Period. Math. Hungar. 2 (1972), 295–299
  \bibitem{drc}
  J. Fox and B. Sudakov, \textit{Dependent Random Choice}, Random Structures Algorithms 38 (2011),
  68–99.
  \bibitem{stuff}
  J. Fox and B. Sudakov, \textit{Unavoidable patterns}, J. Combin. Theory Ser. A 115 (2008), no. 8,
  1561–1569.

  \bibitem{Galvin}
  F. Galvin, Partition theorems for the real line, Notices Amer. Math. Soc. 15 (1968), 660.

  \bibitem{kechris}
  A.S. Kechris, Classical  descriptive  set theory, Graduate Texts in Mathematics,  vol.  156,  Springer-Verlag,  New  York, 1995.
  \bibitem{KoSoTu} T. K\H{o}vari, V. T. S\'os, and P. Tur\'an. \textit{On a problem of K. Zarankiewicz}, Colloquium Math. 3 (1954), 50--57.

  \bibitem{sier}
  W. Sierpiński, \textit{Sur un problème de la théorie des relations}. Annali della Scuola Normale Superiore di Pisa - Classe di Scienze, Serie 2, Volume 2 (1933) no. 3, pp. 285-287.
  \bibitem{kwan}
  M. Kwan, B. Sudakov, and T. Tran, \textit{Anticoncentration for subgraph statistics}, Journal of the London Mathematical Society 99.3 (2019), 757-777.
  \bibitem{promel} H. J. Prömel, \textit{Ramsey Theory for Discrete Structures}, Springer, New York, 2013.
\end{thebibliography}
\end{document}